\documentclass[a4paper,10pt,twoside,english]{scrartcl}
\usepackage[utf8x]{inputenc}
\usepackage{mathrsfs}  
\usepackage{dsfont} 
\usepackage[T1]{fontenc}
\usepackage{lmodern}
\usepackage{graphicx}
\usepackage{enumerate}
\usepackage{setspace}
\usepackage{color}
\usepackage{xspace} 
\usepackage{amsmath,amssymb,amsthm,mathtools}
\usepackage{blkarray}
\usepackage[subnum]{cases}
\usepackage{subcaption}

\usepackage[unicode=true]{hyperref}

\newcommand{\titel}{Singularities of the density of states of random Gram matrices}
\title{\titel} 
\author{
Johannes Alt\footnote{\hspace{0.15cm}Partially funded by ERC Advanced Grant RANMAT No. 338804.\newline Date: \today}
\addtocounter{footnote}{-1}\addtocounter{Hfootnote}{-1}
\\
{\small \begin{tabular}{c}{IST Austria}\\{johannes.alt@ist.ac.at} \end{tabular}} }
\date{}

\hypersetup{
	pdftitle={\titel},
	pdfauthor={Johannes Alt},
	colorlinks={false},
	pdfborderstyle={/S/U/W 1}
}

\numberwithin{equation}{section}

\newtheoremstyle{test}
  {}
  {}
  {\itshape}
  {}
  {\bfseries}
  {.}
  { }
  {}

\newtheoremstyle{testDef}
  {}
  {}
  {}
  {}
  {\bfseries}
  {.}
  { }
  {}

\theoremstyle{testDef}
\newtheorem{defi}{Definition}[section]
\newtheorem{assums}[defi]{Assumptions}

\newtheorem{rem}[defi]{Remark}

\newtheorem*{rem*}{Remark}   
\newtheorem*{ex*}{Example}   
\newtheorem*{def*}{Definition}

\theoremstyle{test}
\newtheorem{thm}[defi]{Theorem}
\newtheorem{lem}[defi]{Lemma}

\newtheorem{convention}[defi]{Convention}
\newtheorem{pro}[defi]{Proposition}
\newtheorem*{pro*}{Proposition} 
\newtheorem*{coro*}{Corollary}
\newtheorem*{thm*}{Theorem}

\newcommand{\ol}[1]{\overline{#1} \!\,} 

\newcommand{\ord} {\mathcal{O}}

\renewcommand{\P}{\mathbb{P}}

\newcommand{\ii}{\mathrm{i}} 

\newcommand{\abs}[1]{\lvert #1 \rvert}

\newcommand{\absB}[1]{\Big\lvert #1 \Big\rvert}
\newcommand{\absbb}[1]{\bigg\lvert #1 \bigg\rvert}

\newcommand{\norm}[1]{\lVert #1 \rVert}

\newcommand{\avg}[1]{\langle #1 \rangle}

\newcommand{\avgB}[1]{\Big\langle #1 \Big\rangle}

\newcommand{\avga}[1]{\left\langle #1 \right\rangle}

\newcommand{\scalar}[2]{\langle{#1} \mspace{2mu}, {#2}\rangle}

\newcommand{\scalara}[2]{\left\langle{#1} \,\mspace{2mu},\, {#2}\right\rangle}

\DeclareMathOperator*{\spec}{Spec}						

\newcommand{\genarg} {{\,\cdot\,}}  

\newcommand{\meas}{\pi}
\newcommand{\measone}{\pi_1}
\newcommand{\meastwo}{\pi_2}
\newcommand{\measr}{\frac{\measone(\Xfone)}{\meastwo(\Xftwo)}} 

\newcommand{\brho}{\boldsymbol \rho}

\newcommand{\dens}{\nu}

\newcommand{\Sone}{\mathcal S_1}
\newcommand{\Stwo}{\mathcal S_2}

\newcommand{\BF}{\mathscr B}
\newcommand{\BFone}{{{\mathscr B}_1}}

\newcommand{\BFtwo}{{{\mathscr B}_2}}

\newcommand{\posset}{\mathfrak P} 
\newcommand{\Xf}{\mathfrak X}
\newcommand{\Xfone}{{\Xf_1}}
\newcommand{\Xftwo}{{\Xf_2}}

\newcommand{\emin}{\ensuremath{\boldsymbol{e}_-}}

\newcommand{\df}{\boldsymbol d}

\newcommand{\wf}{\boldsymbol w}
\newcommand{\uf}{\boldsymbol u}

\newcommand{\ffp}{\boldsymbol f_+}
\newcommand{\ffm}{\boldsymbol f_-}
\newcommand{\gf}{\boldsymbol g}
\newcommand{\yf}{\boldsymbol y}

\newcommand{\bb}{\boldsymbol b}

\newcommand{\mf}{\boldsymbol m}
\newcommand{\pf}{\boldsymbol p}

\newcommand{\Bf}{\boldsymbol B} 

\newcommand{\Ff}{\boldsymbol F}

\newcommand{\Hf}{\boldsymbol H}

\newcommand{\Qfp}{{{\boldsymbol Q}_+}}

\newcommand{\Qfpm}{{{\boldsymbol Q}_\pm}}

\newcommand{\Sf}{\boldsymbol S}

\newcommand{\R}{\mathbb{R}}  
\ifdefined\C\renewcommand{\C}{\mathbb{C}}\else\newcommand{\C}{\mathbb{C}}\fi 
\renewcommand{\Im}{\mathrm{Im}\,} 
\renewcommand{\Re}{\mathrm{Re}\,} 
\newcommand{\N}{\mathbb{N}}  
\newcommand{\E}{\mathbb{E}}  
\renewcommand{\P}{\mathbb{P}}  
\newcommand{\di}{\mathrm{d}} 
\newcommand{\sign}[1]{\mathrm{sign}(#1)} 
\newcommand{\eps}{\varepsilon} 
\newcommand*{\defeq}{\mathrel{\vcenter{\baselineskip0.5ex \lineskiplimit0pt\hbox{\scriptsize.}\hbox{\scriptsize.}}}=}

\newcommand{\pt}{\partial}

\DeclareMathOperator{\supp}{supp}

\newcommand{\Rnon}{\ensuremath{\R_{0}^+}}

\newcommand{\normtwo}[1]{\lVert #1 \rVert_{2}}
\newcommand{\normtwoa}[1]{\left\lVert #1 \right\rVert_{2}}

\newcommand{\norminf}[1]{\lVert #1 \rVert_{\infty}}

\DeclareMathOperator{\Gap}{Gap}

\newcommand{\Hb}{\mathbb H}

\newcommand{\deltamf}{{\tilde{\delta}}}
\newcommand{\HbdS}{{\mathbb H}_{\deltamf}^\Sigma}
\newcommand{\HbdSclosed}{\smash{\overline{\mathbb H}}\vphantom{\mathbb H}_{\deltamf}^\Sigma}

\newcommand{\Dquad}{\mathcal D}

\renewcommand{\char}{\ensuremath{\chi}} 
	
\newcommand{\measprop}{\hyperlink{lab:measprop}{\textbf{(A1)}}\xspace}
\newcommand{\Ltwoinf}{\hyperlink{lab:Ltwoinf}{\textbf{(A4)}}\xspace}
\newcommand{\lowerbound}{\hyperlink{lab:lowerbound}{\textbf{(A2)}}\xspace}
\newcommand{\mbounded}{\hyperlink{lab:mbounded}{\textbf{(A3)}}\xspace}

\newcommand{\mfbounded}{\hyperlink{lab:mfbounded}{\textbf{(C2)}}\xspace}

\newcommand{\ranupbound}{\hyperlink{lab:ranupbound}{\textbf{(B1)}}\xspace}
\newcommand{\ranmombound}{\hyperlink{lab:ranmombound}{\textbf{(B2)}}\xspace}

\begin{document}

\maketitle
\vspace*{-1.3cm}
\begin{abstract}
For large random matrices $X$ with independent, centered entries but not necessarily identical variances,  
 the eigenvalue density of $XX^*$ is well-approximated by a deterministic measure on $\R$.
We show that the density of this measure has only square and cubic-root singularities away from zero. 
We also extend the bulk local law in \cite{AltGram} to the vicinity of these singularities. 
\end{abstract}


\section{Introduction}

The empirical eigenvalue density or \emph{density of states} of many large random matrices is well-approximated by a deterministic probability measure, the \emph{self-consistent density of states}.
If $X$ is a $p \times n$ random matrix with independent, centered entries of identical variances then the limit of 
the eigenvalue density of the \emph{sample covariance matrix} $XX^*$ for large $p$ and $n$ with $p/n$ converging to a constant has been identified by Marchenko and Pastur in \cite{MarcenkoPastur1967}. 
However, some applications in wireless communication require understanding the spectrum of $XX^*$ 
without the assumption of identical variances of the entries of $X=(x_{kq})_{k,q}$~\cite{wirelesscommunication,hachem2007,TulinoVerdu}.
In this case, the matrix $XX^*$ is a \emph{random Gram matrix}.

For constant variances, the self-consistent density of states is obtained by solving a scalar equation for its Stieltjes transform, the \emph{scalar Dyson equation}. 
In case the variances $s_{kq} \defeq \E \abs{x_{kq}}^2$ depend nontrivially on $k$ and $q$, the self-consistent density of states is obtained from the 
solution $m(\zeta)=(m_1(\zeta), \ldots, m_p(\zeta))\in\Hb^p$ of the \emph{vector Dyson equation} \cite{girko2012theory}
\begin{equation} \label{eq:Dyson_Gram_matrix}
-\frac{1}{m_k(\zeta)} =  \zeta - \sum_{q=1}^n s_{kq} \Big(1+ \sum_{l=1}^p s_{lq}m_l(\zeta)\Big)^{-1} \qquad\quad \text{for all } k \in [p],
\end{equation}
for all $\zeta \in \Hb$. Here, we introduced $\Hb \defeq \{ \zeta \in \C \colon \Im \zeta >0 \}$ and 
$[p] \defeq \{1, \ldots, p\}$. Indeed, the average $\avg{m(\zeta)}_1 \defeq p^{-1}\sum_{k=1}^p m_k(\zeta)$ is the Stieltjes transform of the self-consistent density of states
denoted by $\avg{\dens}_1$.
If the limit of $\avg{\dens}_1$ as $p,n\to\infty$ exists then it can be studied via an infinite-dimensional version of \eqref{eq:Dyson_Gram_matrix} 
 (see \eqref{eq:Gram_equation} below).


For Wigner-type matrices, i.e., Hermitian random matrices with independent (up to the Hermiticity constraint), centered entries, the analogue of \eqref{eq:Dyson_Gram_matrix} 
is a quadratic vector equation (QVE) in the language of \cite{AjankiQVE,AjankiCPAM}. 
In these papers, finite and infinite-dimensional versions of the QVE have been extensively studied to analyze the self-consistent density of states whose Stieltjes transform is the average of the solution to the QVE.
 The authors show that the self-consistent density of states has a $1/3$-Hölder continuous density. Except for finitely many square-root and cubic-root singularities this density is real-analytic. 
The square-root behaviour emerges solely at the edges of the connected components of the support of the self-consistent density of states, 
whereas the cubic-root singularities lie inside these components. The detailed stability analyis in \cite{AjankiQVE} is then used in \cite{Ajankirandommatrix} to obtain 
the local law for Wigner-type matrices. A \emph{local law} typically refers to a statement about the convergence of the eigenvalue density to a deterministic measure on a scale slightly above the typical local 
eigenvalue spacing.


For the Dyson equation for random Gram matrices, we obtain away from $\zeta=0$ the same results as mentioned above in the QVE setup. 
Furthermore, we extend our local law for random Gram matrices in \cite{AltGram} to the vicinity of the singularities of the self-consistent density of states.
This can be seen as another instance of the universality phenomenon in random matrix theory. Despite the different structure of Gram and Wigner-type matrices, 
the densities of states of these Hermitian random matrices have the same types of singularities. 
We refer to \cite{AltGram} and the references therein for related results about random Gram matrices. 


There is a close connection between Gram and Wigner-type matrices. 
The Dyson equation, \eqref{eq:Dyson_Gram_matrix}, can be transformed into a QVE in the sense of \cite{AjankiQVE} and the spectrum of $XX^*$ is closely related to that of a Wigner-type matrix 
in the sense of~\cite{Ajankirandommatrix}. 
This is easiest explained on the random matrix level through a special case of the linearization tricks: If $X$ has independent and centered entries then the random matrix 
\begin{equation} \label{eq:def_Hf}
\Hf = \begin{pmatrix} 0 & X \\ X^* & 0 \end{pmatrix}
\end{equation}
is a Wigner-type matrix and the spectra of $\Hf^2$ and $XX^*$ agree away from zero. 
Therefore, instead of trying to analyze \eqref{eq:Dyson_Gram_matrix} and $XX^*$ directly, it is more efficient to study the corresponding QVE and Wigner-type matrix as in \cite{AltGram}. 
However, owing to the large zero blocks in $\Hf$, its variance matrix is not uniformly primitive (see \textbf{A3} in \cite{AjankiQVE}), a key assumption for the analysis in \cite{AjankiQVE}.
Indeed, the stability operator of the QVE possesses an additional unstable direction $\ffm$, which has to be treated separately. 
In \cite{AltGram}, this study has been conducted in the bulk spectrum and away from the support of $\avg{\dens}_1$, where $\ffm$ did not play an important role at least away from zero. 

In this note, we present a new argument needed in the analysis of the cubic equation (see \eqref{eq:general_cubic} below) describing the stability of the QVE 
close to its singularities in order to incorporate the additional unstable direction. In fact, the analysis of the cubic equation in \cite{AjankiQVE} heavily relies on the uniform primitivity of the 
variance matrix. Adapting this argument to the current setup cannot exclude 
that the coefficients of the cubic and the quadratic term in the cubic equation vanish at the same time due to the presence of~$\ffm$. 
A nonvanishing cubic or quadratic coefficient is however absolutely crucial 
for the cubic stability analysis in \cite{AjankiQVE}. 
Otherwise not only square-root or cubic-root but also higher order singularities would emerge. 
Our main novel ingredient, a very detailed analysis of these coefficients, actually excludes this scenario.
With this essential new input, the regularity and the singularity structure of \eqref{eq:Dyson_Gram_matrix} as well as the local law for $XX^*$ follow by correctly combining the arguments in 
\cite{AjankiQVE,Ajankirandommatrix,AltGram}.


\paragraph{Acknowledgement} The author is very grateful to László Erd{\H o}s for many fruitful discussions and many valuable suggestions. The author would also like to thank Torben Krüger for several helpful conversations.

\section{Main results}

\subsection{Structure of the solution to the Dyson equation} \label{subsec:result_equation}

Let $(\Xfone, \Sone, \measone)$ and $(\Xftwo, \Stwo, \meastwo)$ be two finite measure spaces such that  $\measone(\Xfone)$ and $\meastwo(\Xftwo)$ 
are strictly positive. 
Moreover, we denote the spaces of bounded and measurable functions on $\Xfone$ and $\Xftwo$ by
\[ \BF_i \defeq \left\{ u \colon \Xf_i \to \C : \norminf{u} \defeq \sup_{x\in\Xf_i} \abs{u(x)} < \infty \right\}\]
for $i = 1, 2$. We consider $\BFone$ and $\BFtwo$ equipped with the supremum norm $\norminf{\genarg}$. 
We denote the induced operator norms by $\norm{\genarg}_{\BFone\to\BFtwo}$ and $\norm{\genarg}_{\BFtwo\to\BFone}$. 
For $u \in\BFone$, we write $u_k =u(k)$ for $k \in \Xfone$.  We use the same notation for~$v\in\BFtwo$.

Let $s\colon \Xfone \times \Xftwo \to \Rnon, s(k,q) = s_{kq}$ be a measurable nonnegative function such that 
\begin{equation} \label{eq:kernel_bound_inf_to_inf}
 \sup_{k \in \Xfone} \int_\Xftwo s_{kq} \meastwo(\di q) < \infty, \qquad  \sup_{q \in \Xftwo} \int_\Xfone s_{kq} \measone(\di k) < \infty.
\end{equation}
 We define the bounded linear operators $S \colon \BFtwo \to \BFone$ and $S^t \colon \BFone \to \BFtwo$ through 
\begin{equation} \label{eq:def_S_St}
 (Sv)_k =\int_\Xftwo s_{kr} v_r \meastwo(\di r), ~~ k \in \Xfone, ~ v \in \BFtwo, 
\quad  (S^tu)_q =\int_\Xfone s_{lq} u_l \measone(\di l), ~~ q \in \Xftwo, ~ u \in \BFone.
\end{equation}
We are interested in the solution $m \colon \Hb \to \BFone$ of the \emph{Dyson equation}
\begin{equation} \label{eq:Gram_equation}
 - \frac{1}{m(\zeta)} = \zeta - S \frac{1}{1 + S^t m(\zeta)}, 
\end{equation}
for $\zeta \in \Hb$, which satisfies $\Im m(\zeta) >0$ for all $\zeta \in \Hb$. 

\begin{pro}[Existence and Uniqueness] \label{pro:existence_and_uniqueness}
If \eqref{eq:kernel_bound_inf_to_inf} holds true then there is a unique function $m\colon \Hb \to \BFone$ 
satisfying \eqref{eq:Gram_equation} and $\Im m(\zeta) >0$ for all $\zeta \in \Hb$. 
Moreover, $m \colon \Hb \to \BFone$ is analytic. 
For each $k \in \Xfone$, there is a unique probability measure $\dens_k$ on $\R$ such that $m_k$ is the Stieltjes transform of $\dens_k$, i.e., 
\begin{equation} \label{eq:def_dens}
 m_k(\zeta) = \int_{0}^{\infty} \frac{1}{E-\zeta} \dens_k(\di E) 
\end{equation}
for all $\zeta \in \Hb$. The support of $\dens_k$ is contained in $[0,\Sigma]$ for each $k \in \Xfone$, where
\begin{equation} \label{eq:def_Sigma_support}
\Sigma \defeq 4 \max\left\{\norm{S}_{\BFtwo \to\BFone},\norm{S^t}_{\BFone \to\BFtwo} \right\}. 
\end{equation}
\end{pro}

Further assumptions on $\measone$, $\meastwo$ and $S$ will yield a more detailed understanding of the measures $\dens_k$. 
To formulate these assumptions, we 
introduce  the averages of $u \in \BFone$ 
and $v \in \BFtwo$ through
\[\avg{u}_1 = \frac{1}{\measone(\Xfone)} \int_\Xfone u_k \measone(\di k), \qquad 
\avg{v}_2 = \frac{1}{\meastwo(\Xftwo)} \int_\Xftwo v_q \meastwo(\di q).
\] 
Additionally, we set $\norm{u}_t \defeq \avg{\abs{u}^t}_1^{1/t}$ and $\norm{v}_t \defeq \avg{\abs{v}^t}_2^{1/t}$ for $u \in\BFone$, $v \in \BFtwo$ and $t\geq 1$. 
Moreover, for $k \in \Xfone$ and $q \in \Xftwo$, we define the functions $S_k \colon \Xftwo\to \Rnon, ~ S_k(r) = s_{kr} $ 
and $(S^t)_q \colon \Xfone \to \Rnon,~ (S^t)_q(l) = s_{lq}$.
We call $S_k$ and $(S^t)_q$ the \emph{rows} and \emph{columns} of $S$, respectively. 

\begin{assums} \label{assums:Gram_assumptions}
\begin{enumerate}
\item[\measprop] \hypertarget{lab:measprop}{} The total measures $\measone(\Xfone)$ and $\meastwo(\Xftwo)$ are comparable, i.e., there are constants $0< \pi_* < \pi^*$ such that 
\[ \pi_* \leq \measr\leq \pi^*. \]
\item[\lowerbound] \hypertarget{lab:lowerbound}{} The operators $S$ and $S^t$ are irreducible in the sense that there are $L_1$, $L_2 \in \N$ and $\kappa_1, \kappa_2 >0$ 
such that 
 \[ \left((SS^t)^{L_1}u\right)_{k} \geq \kappa_1\avg{u}_1, \quad \left((S^tS)^{L_2}v\right)_{q} \geq \kappa_2\avg{v}_2, \]
for all $u \in \BFone$, $v \in \BFtwo$ satisfying $u\geq 0$ and $v\geq 0$ and for all $k \in \Xfone$, $q \in \Xftwo$. 
\item[\mbounded] \hypertarget{lab:mbounded}{} The rows and columns of $S$ are sufficiently close to each other in the sense that
there is a continuous strictly monotonically decreasing function $\gamma \colon (0,1] \to \Rnon$ such that $\lim_{\eps\downarrow 0} \gamma(\eps) = \infty$ and 
for all $\eps \in (0,1]$, we have 
\[ 
\gamma(\eps) \leq \min \bigg\{ \inf_{k \in \Xfone} \frac{1}{\measone(\Xfone)} \int_\Xfone \frac{\measone(\di l)}{\eps + \normtwo{S_k-S_{l}}^2},  
\inf_{q \in \Xftwo} \frac{1}{\meastwo(\Xftwo)} \int_\Xftwo \frac{\meastwo(\di r)}{\eps + \normtwo{(S^t)_q-(S^t)_{r}}^2}\bigg\}.
\] 
\item[\Ltwoinf] \hypertarget{lab:Ltwoinf}{} The operators $S$ and $S^t$ map square-integrable functions continuously to bounded functions, i.e., 
there are constants $\Psi_1,\Psi_2>0$ such that 
\[ \norm{S}_{L^2(\meastwo/\meastwo(\Xftwo)) \to \BFone} \leq \Psi_1, \qquad \norm{S^t}_{L^2(\measone/\measone(\Xfone)) \to \BFtwo} \leq \Psi_2. \]
\end{enumerate}
\end{assums}

Our estimates will be uniform in all models that satisfy Assumptions~\ref{assums:Gram_assumptions} with the same constants. 
Therefore, the constants $\pi_*$, $\pi^*$ from \measprop, $L_1$, $L_2$, $\kappa_1$, $\kappa_2$ from \lowerbound,
the function $\gamma$ from \mbounded and $\Psi_1$, $\Psi_2$ from \Ltwoinf are called \emph{model parameters}.
We refer to Remark~\ref{rem:Hoelder_to_boundedness} below for an easily checkable sufficient condition for \mbounded.
We now state our main result about the regularity and the possible singularities of $\dens_k$ defined in \eqref{eq:def_dens}.

\begin{thm}\label{thm:prop_selfcon_density_states}
 If we assume \measprop~--~\Ltwoinf then the following statements hold true:
\begin{enumerate}[(i)]
\item (Regularity of $\dens$) There are $\dens^0 \in \BFone$ and $\dens^d \colon \Xfone \times (0,\infty) \to [0,\infty), ~(k,E) \mapsto \dens_k^d(E)$ such that
\begin{equation} \label{eq:def_dens_0_dens_d}
\dens_k(\di E) = \dens_k^0\delta_0(\di E) + \dens_k^d(E) \di E  
\end{equation}
for all $k \in \Xfone$. 
For all $\delta>0$, the function $\dens^d$ is uniformly $1/3$-Hölder continuous on $[\delta, \infty)$, i.e., 
\[ \sup_{k \in \Xfone} \sup_{E_1 \neq E_2,\, E_1, E_2 \geq \delta}\frac{\abs{\dens^d_k(E_1)-\dens^d_k(E_2)}}{\abs{E_1 - E_2}^{1/3}} < \infty. \] 
For all $k \in \Xfone$, we have 
\[ \{ E \in (0,\infty) \colon \avg{\dens^d(E)} >0 \} = \{ E \in (0,\infty)\colon \dens^d_k(E) >0 \}.   \]
We set $\posset \defeq \{ E \in (0,\infty) \colon \avg{\dens^d(E)} >0 \}$. 
For each $\delta >0$, the set $\posset\cap (\delta, \infty)$ is a finite union of open intervals. 
The map $\dens^d \colon (0,\infty)\setminus \pt\posset \to \BFone$ is real-analytic. 
There is $\rho_*>0$ depending only on the model parameters 
and $\delta$ such that the Lebesgue measure of each connected component of $\posset\cap(\delta,\infty)$ is at least $2\rho_*$.
\item (Singularities of $\dens^d$) Fix $\delta>0$. For any $E_0 \in (\pt\posset)\cap (\delta, \infty)$, there are two cases
\begin{itemize} 
\item[~~\emph{CUSP:}] The point $E_0$ is the intersection of the closures of two connected components of $\posset\cap (\delta,\infty)$ and $\dens^d$ has a cubic root singularity at $E_0$, i.e., 
there is $c \in \BFone$ satisfying $\inf_{k\in\Xfone} c_k >0$ such that 
\[ \dens_k^d(E_0 + \lambda) = c_k \abs{\lambda}^{1/3} + \ord(\abs{\lambda}^{2/3}), \qquad \lambda \to 0. \] 
\item[~~\emph{EDGE:}] The point $E_0$ is the left or right endpoint of a connected component of $\ol{\posset}\cap(\delta, \infty)$ and $\dens^d$ has a square root singularity at $E_0$, i.e., 
there is $c \in \BFone$ satisfying $\inf_{k \in\Xfone} c_k >0$ such that 
\[ \dens_k^d(E_0 + \theta \lambda) = c_k \lambda^{1/2} + \ord(\lambda), \qquad \lambda\downarrow 0, \]
where $\theta = + 1$ if $E_0$ is a left endpoint of $\posset$ and $\theta = -1$ if $E_0$ is a right endpoint. 
\end{itemize}
\end{enumerate}
\end{thm}

\begin{figure}[ht!]
\begin{subfigure}{.5\textwidth} 
\includegraphics[width=.95\textwidth]{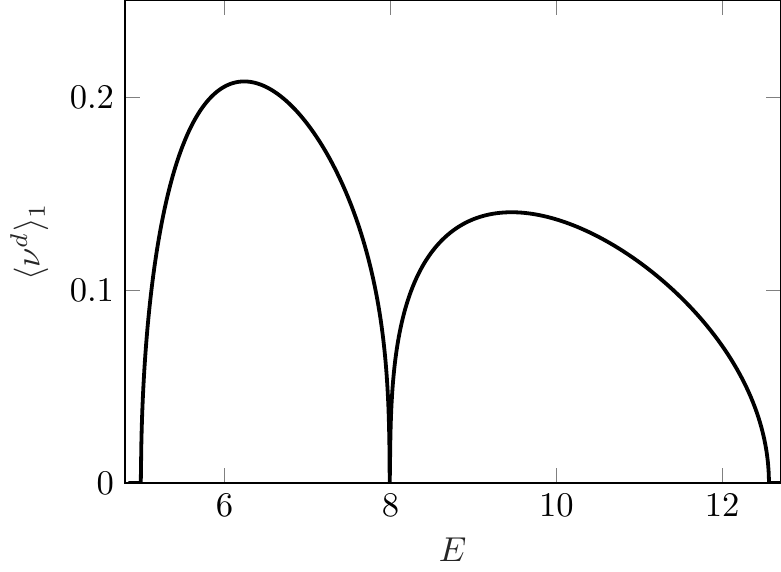}
\caption{Self-consistent density of states $\avg{\dens^d}_1$.}
\end{subfigure}
\begin{subfigure}{.5\textwidth}
\[ S^t = 
\begin{array}{c l} 
\vspace{-1.0cm} & \\ 
\xleftrightarrow[\phantom{\hspace{2.3cm}}]{\displaystyle\kappa_c n}%
\xleftrightarrow[\phantom{\hspace{2.3cm}}]{\displaystyle\kappa_c n}& \\
\begin{array}{|c|}\hline \hspace{1cm} 6 \hspace{1cm}\\ \hline \hspace{1cm} 4 \hspace{1cm} \\ \hline \end{array} 
\begin{array}{|c|}\hline \hspace{1cm} 4 \hspace{1cm}\\ \hline \hspace{1cm} 3 \hspace{1cm}\\\hline\end{array}& 
\left.
\phantom{\begin{array}{|c|}\hline \\\hline \\\hline\end{array}} \hspace{-0.5cm}
 \right\updownarrow{n}
\end{array}
\]
\caption{Variance profile.}
\end{subfigure}
\caption{Example of a self-consistent density of states with variance profile $S$. It has square-root edges at the left and right endpoint of its support and a cubic cusp at $E \approx 8$.}
\end{figure}

\begin{rem}[Piecewise Hölder-continuous rows and columns of $S$ imply \mbounded] \label{rem:Hoelder_to_boundedness}
Let $\Xfone$ and $\Xftwo$ be two nontrivial compact intervals in $\R$ and $\measone$ and $\meastwo$ the Lebesgue measures.
In this case, \mbounded holds true if the maps $k \mapsto S_k$ and $r \mapsto (S^t)_r$ are piecewise $1/2$-Hölder continuous in the sense that
 there are two finite partitions $(I_\alpha)_{\alpha \in A}$ and $(J_\beta)_{\beta \in B}$ of $\Xfone$ and $\Xftwo$, respectively, such that, for all $\alpha \in A$ and $\beta\in B$, we have
 \[ \normtwo{S_k - S_l} \leq C_\alpha \abs{k-l}^{1/2}~\text{ for } k, l \in I_\alpha, \qquad \normtwo{(S^t)_q - (S^t)_r} \leq D_\beta \abs{q-r}^{1/2}~\text{ for } q, r \in J_\beta.  \]
There is a similar condition for \mbounded if $\Xfone = [p]$ and $\Xftwo =[n]$ for some $p,n \in\N$ and the measures $\measone$ and $\meastwo$ are the (unnormalized) counting measures on $[p]$ and $[n]$, respectively. 
\end{rem}

\subsection{Local law for random Gram matrices} \label{subsec:result_Gram}

In this subsection, we state our results on random Gram matrices. We now set $\Xfone = [p]$, $\Xftwo = [n]$ as well as $\measone$ and $\meastwo$ the (unnormalized) counting 
measures on $[p]$ and $[n]$, respectively. In particular, $\measone(\Xfone) = p$ and $\meastwo(\Xftwo) = n$.

\begin{assums}\label{assums:random_matrix} Let $X=(x_{kq})_{k,q}$ be a $p\times n$ random matrix with independent, centered entries and variance matrix $S = (s_{kq})_{k,q}$, i.e., 
$\E\, x_{kq} = 0$  and $s_{kq} \defeq \E \abs{x_{kq}}^2$ for $k\in [p]$, $q \in [n]$.  
Moreover, we assume that \measprop, \lowerbound and \mbounded in Assumptions \ref{assums:Gram_assumptions} and the following conditions are satisfied.
\begin{enumerate}
\item[\ranupbound] \hypertarget{lab:ranupbound}{} The variances are bounded in the sense that there exists $s^* >0$ such that 
\[ s_{kq} \leq \frac{s^*}{p+n} \qquad\qquad \text{for }k\in [p],~ q\in [n]. \]
\item[\ranmombound] \hypertarget{lab:ranmombound}{} All entries of $X$ have bounded moments in the sense that
there are $\mu_m >0$ for $m \geq 3$ such that 
\[ \E \abs{x_{kq}}^m \leq \mu_m s_{kq}^{m/2}\qquad\qquad \text{for all }k \in[p], ~q \in[n].  \]
\end{enumerate}
\end{assums}

\noindent The sequence $(\mu_m)_{m\geq 3}$ in \ranmombound is also considered a model parameter. 

Since \ranupbound implies \Ltwoinf, we can apply Theorem \ref{thm:prop_selfcon_density_states}. 
By its first part, for every $\delta>0$, there are $\alpha_1, \ldots, \alpha_K$, $\beta_1, \ldots, \beta_K \in [\delta, \infty)$ for some $K \in \N$ such that 
\[ \supp \avgB{\dens^d|_{[\delta, \infty)}}_1 =\bigcup_{i=1}^K [\alpha_i, \beta_i], \qquad \alpha_j  < \beta_j < \alpha_{j+1} \]
and $\rho_*>0$ depending only on the model parameters and $\delta$ such that $\beta_i-\alpha_i \geq 2\rho_*$ for all $i\in [K]$. 
For $\rho \in [0,\rho_*)$, we introduce the \emph{local gap size} $\Delta_\rho$ via
\begin{equation} \label{eq:def_Delta_rho}
 \Delta_\rho(E) \defeq \begin{cases} \alpha_{i+1} - \beta_i, & \text{if } \beta_i -\rho \leq E \leq \alpha_{i+1} + \rho \text{ for some }i \in [K],\\
1, & \text{if } E \leq \alpha_1 + \rho \text{ or }E \geq \beta_K - \rho, \\ 0, &\text{otherwise.} \end{cases} 
\end{equation}
For $\delta,\gamma>0$, we define the spectral domain $\mathbb{D}_{\delta, \gamma} \defeq \big\{ \zeta \in\Hb \colon \abs{\zeta} \geq \delta,~ \Im \zeta \geq p^{-1 + \gamma} \big\}$.
We introduce the resolvent $R(\zeta) \defeq (XX^*-\zeta)^{-1}$ of $XX^*$ at $\zeta \in\Hb$ and denote its entries by $R_{kl}(\zeta)$ for $k, l \in [p]$.

\begin{thm}[Local law for Gram matrices] \label{thm:local_law}
Let Assumptions \ref{assums:random_matrix} hold true.
Fix $\delta>0$ and $\gamma \in (0,1)$. 
Then there is $\rho \in (0,\rho_*)$ depending only on the model parameters and $\delta$ 
such that if we define $\kappa = \kappa^{(p)} \colon \Hb \to (0,\infty]$ through
\[ \kappa(\zeta) = \big(\Delta_\rho(\Re \zeta)^{1/3} + \avg{\Im m(\zeta)}\big)^{-1} \]
then, for each $\eps>0$ and $D>0$, there is a constant $C_{\eps, D}>0$ such that
\begin{subequations} \label{eq:local_laws}
\begin{equation} \label{eq:local_law_entry}
\P \left( \sup_{\substack{\zeta \in \mathbb{D}_{\delta, \gamma}\\ k,l\in [p]}}p^{-\eps}\absB{R_{kl}(\zeta) - m_k(\zeta) \delta_{kl}}\leq \sqrt{\frac{\avg{\Im m(\zeta)}}{p\Im \zeta}} + \min\bigg\{ \frac{1}{\sqrt{p\Im\zeta}}, \frac{\kappa(\zeta)}
{p\Im \zeta}\bigg\} \right) \geq 1 - \frac{C_{\eps,D}}{p^D}. 
\end{equation} 
Furthermore, for any $\eps>0$ and $D>0$, there is a constant $C_{\eps,D}>0$ such that, for any deterministic vector $w \in \C^p$ satisfying $\max_{k \in [p]} \abs{w_k} \leq 1$, we have 
\begin{equation} \label{eq:local_law_average}
\P \left( \sup_{\zeta \in \mathbb{D}_{\delta, \gamma}} \absbb{\frac{1}{p} \sum_{k=1}^p w_k \Big( R_{kk}(\zeta) - m_k(\zeta)\Big)} \leq p^\eps \min\bigg\{ \frac{1}{\sqrt{p\Im \zeta}}, \frac{\kappa(\zeta)}{p\Im\zeta}\bigg\} \right) \geq 1 - \frac{C_{\eps,D}}{p^D}.
\end{equation} 
\end{subequations}
The constant $C_{\eps, D}$ in \eqref{eq:local_laws} depends only on the model parameters as well as $\delta$ and $\gamma$ in addition to $\eps$ and $D$. 
\end{thm}

\begin{rem}
\begin{enumerate}[(i)]
\item (Corollaries of the local law)
In the same way as in \cite{Ajankirandommatrix} and in \cite{AltGram}, the standard corollaries of a local law -- convergence of cumulative distribution function, rigidity of eigenvalues, 
anisotropic law and delocalization of eigenvectors -- may be proven. 
\item (Local law in the bulk and away from $\supp \dens$) 
In the bulk, Theorem \ref{thm:local_law} has already been proven in \cite{AltGram}.
Away from $\supp \dens$, the convergence rate in \eqref{eq:local_law_entry} and \eqref{eq:local_law_average} can be improved and 
thus the condition $\Im \zeta \geq p^{-1 + \gamma}$ can be removed. See~\cite{AltGram} 
for Gram matrices and \cite{AltKronecker} for Kronecker matrices. 
\item (Local law close to zero)
Strengthening the assumption \lowerbound, 
we have proven the local law close to zero in the cases, $n=p$ and $\abs{p-n} \geq c n$, in \cite{AltGram}.
\end{enumerate}
\end{rem}

\section{Quadratic vector equation}

In this section, we translate \eqref{eq:Gram_equation} into a quadratic vector equation of \cite{AjankiQVE} (see \eqref{eq:QVE} below) and show that Proposition \ref{pro:existence_and_uniqueness} trivially
follows from \cite{AjankiQVE}. However, the singularity analysis in \cite{AjankiQVE} has to be changed essentially due to the violation of the uniform primitivity condition, \textbf{A3} in \cite{AjankiQVE}, on $\Sf$ 
(cf. \eqref{eq:def_Sf} below) in our setup. 

Let $\Xf \defeq \Xfone \sqcup \Xftwo$ be the disjoint union of $\Xfone$ and $\Xftwo$
and $\meas$ the probability measure defined through
\[ \meas(A\sqcup B) = \big(\measone(\Xfone) + \meastwo(\Xftwo)\big)^{-1}\big( \measone(A) + \meastwo(B)\big), \qquad \text{for } A\subset \Xfone,~ B \subset \Xftwo. \]
Moreover, we denote the set of bounded measurable functions $\Xf \to \C$ by
$\BF \defeq \{ \wf \colon \Xf \to \C \colon \norminf{\wf} \defeq \sup_{x\in\Xf} \abs{\wf(x)} < \infty\}$
with the supremum norm $\norminf{\genarg}$.
Finally, on $\BF = \BFone \oplus \BFtwo$, we define the linear operator $\Sf \colon \BF \to \BF$ through
\begin{equation} \label{eq:def_Sf}
\Sf \defeq \begin{pmatrix} 0 & S \\ S^t & 0 \end{pmatrix}\text{,} \qquad\qquad \text{i.e.,} \quad \Sf \wf = S(\wf|_\Xftwo) + S^t(\wf|_\Xfone) \quad \text{for }\wf \in \BF.
\end{equation}
Here, we consider $S(\wf|_\Xftwo)$ and $S^t(\wf|_\Xfone)$ as functions $\Xf \to \C$, extended by zero outside of $\Xfone$ and $\Xftwo$, respectively. 
Instead of \eqref{eq:Gram_equation}, we study the quadratic vector equation (QVE)
\begin{equation} \label{eq:QVE}
 -\frac{1}{\mf} = z + \Sf \mf 
\end{equation}
for $z \in \Hb$. Here, we used the change of variables $z^2 = \zeta$. 
We now explain how $\mf$ and $m$ are related. 
If $\mf$ is a solution of \eqref{eq:QVE} then $m_1 \defeq \mf|_\Xfone$
and $m_2 \defeq \mf|_\Xftwo$ satisfy $-m_1^{-1} = z + Sm_2$ and $-m_2^{-1} = z + S^t m_1$. 
Solving the second equation for $m_2$, plugging the result into the first relation and choosing
$z = \sqrt{\zeta} \in \Hb$, we see that $m$ defined through 
\begin{equation} \label{eq:transform_QVE_to_Gram}
m(\zeta) = \frac{m_1(\sqrt\zeta)}{\sqrt\zeta}
\end{equation}
for $\zeta \in \Hb$ is a solution of \eqref{eq:Gram_equation}. If $\mf$ has positive imaginary part then $m$ as well. 

For $\uf\in\BF$, we write $\uf_x \defeq \uf(x)$ with $x \in\Xf$. 
For $\uf, \wf \in \BF$, we denote the scalar product of $\uf$ and $\wf$ and the average of $\uf$ by
\begin{equation} \label{eq:def_scalar_product}
 \scalar{\uf}{\wf} \defeq \int_\Xf \overline{\uf_x}\; \wf_x \meas(\di x), \quad \avg{\uf} \defeq \scalar{1}{\uf} = \int_\Xf \uf_x \meas(\di x). 
\end{equation}
We also introduce the Hilbert space $L^2(\meas) \defeq \{ \uf \colon \Xf \to \C \colon \scalar{\uf}{\uf} < \infty \}$.
The operator $\Sf$ is symmetric on $\BF$ with respect to $\scalar{\genarg}{\genarg}$ and positivity preserving, as $s_{kr} \geq 0$ for all $k \in \Xfone$ and $r \in \Xftwo$.
Therefore, by Theorem 2.1 in \cite{AjankiQVE}, there exists $\mf \colon \Hb \to \BF$ which satisfies \eqref{eq:QVE} for all $z \in \Hb$. 
This function is unique if we require that the solution of \eqref{eq:QVE} satisfies $\Im \mf(z) >0$ for $z \in \Hb$. 
Moreover, $\mf\colon \Hb \to \BF$ is analytic and,
for all $z \in\Hb$, we have 
\begin{equation} \label{eq:L2_bound}
 \normtwo{\mf(z)} \leq 2\abs{z}^{-1}. 
\end{equation}
Furthermore, for all $x \in \Xf$, there are symmetric probability measures $\brho_x$ on $\R$ such that 
\begin{equation} \label{eq:mf_Stieltjes_representation}
 \mf_x(z) = \int_\R \frac{1}{\tau-z} \brho_x(\di \tau) 
\end{equation}
for all $z \in \Hb$ \cite{AjankiQVE}. That means that $\mf_x$ is the Stieltjes transform of $\brho_x$. 
By (2.7) in \cite{AjankiQVE}, 
 the definition of $\Sigma$ in \eqref{eq:def_Sigma_support} and 
 $\norm{\Sf} = \norm{\Sf}_{\BF \to \BF} = \max\{ \norm{S}_{\BFtwo\to \BFone} , \norm{S^t}_{\BFone\to \BFtwo} \}$,
the support of $\brho_x$ is contained in $[- \Sigma^{1/2}, \Sigma^{1/2}]$. 
\begin{proof}[Proof of Proposition \ref{pro:existence_and_uniqueness}]
The existence of $m$ directly follows from the transform in \eqref{eq:transform_QVE_to_Gram} 
and the existence of $\mf$. The uniqueness of $m$ and the existence of $\dens_k$, $k \in \Xfone$, 
are obtained as in the proof of Theorem 2.1 in \cite{AltGram}. 
\end{proof}

The special structure of $\Sf$ (cf. \eqref{eq:def_Sf}) implies an important symmetry of the solution~$\mf$. We multiply \eqref{eq:QVE} by $\mf$ and take the scalar product of the result with $\emin \in \BF$ defined 
through $\emin(k) = 1$ if $k \in \Xfone$ and $\emin(q) = -1$ if $ q \in \Xftwo$. As $\scalar{\emin}{\mf (\Sf \mf)} = 0$, we have
\begin{equation} \label{eq:scalar_emin_m}
 z \scalar{\emin}{\mf} = - \avg{\emin} = - \frac{\measone(\Xfone)-\meastwo(\Xftwo)}{\measone(\Xfone) + \meastwo(\Xftwo)},
\qquad \text{ for all }z \in \Hb.
\end{equation}

\begin{assums}\label{assums:QVE}
In the remainder of this section, 
we assume that \measprop, \lowerbound, \Ltwoinf and the following condition hold true:
\begin{enumerate}
\item[\mfbounded] \hypertarget{lab:mfbounded}{}
There are $\deltamf>0$ and $\Phi >0$ such that for all $z \in \Hb$ satisfying $\abs{z} \geq \deltamf$, we have
\[ \norminf{\mf(z)} \leq \Phi.\]
\end{enumerate}
\end{assums}

\begin{rem}[Relation between \mbounded and \mfbounded] \label{rem:Gram_assumptions_to_QVE}
By slightly adapting the proofs of Theorem 6.1 (ii) and Proposition 6.6 in \cite{AjankiQVE}, we see that, by \mbounded, 
for each $\deltamf>0$, there is $\Phi_{\deltamf}>0$ such that 
\mfbounded is satisfied with a constant $\Phi\equiv\Phi_{\deltamf}$.
\end{rem}

Since our estimates in this section will be uniform in all 
models that satisfy \measprop, \lowerbound, \Ltwoinf and \mfbounded with the same constants, we introduce the following notion.

\begin{convention}[Comparison relation] 
For nonnegative scalars or vectors $f$ and $g$, we will use the notation $f \lesssim g$ if there is a constant $c>0$, depending only on 
$\pi_*, \pi^*$ in \measprop, $L_1,L_2, \kappa_1, \kappa_2$ in \lowerbound, $\Psi_1, \Psi_2$ in \Ltwoinf as well as $\deltamf$ and $\Phi$ in \mfbounded, such that $f \leq cg$.
Moreover, we write $f \sim g$ if both, $f \lesssim g$ and $f \gtrsim g$, hold true. 
\end{convention}

\subsection{Hölder continuity and analyticity}

We recall $\Sigma$ from \eqref{eq:def_Sigma_support} and introduce the set $\HbdS \defeq \big\{ z \in \Hb \colon 2 \deltamf \leq \abs{z} \leq 10 \Sigma^{1/2}\big\}$
 and its closure $\HbdSclosed$.

\begin{pro}[Regularity of $\mf$] \label{pro:properties_of_mf}
Assume \measprop, \lowerbound, \Ltwoinf and \mfbounded.
\begin{enumerate}[(i)]
\item The restriction $\mf \colon \HbdS \to \BF$ is uniformly $1/3$-Hölder continuous, i.e., 
\begin{equation} \label{eq:Hoelder_continuity_mf}
 \norminf{\mf(z) - \mf(z')} \lesssim \abs{z-z'}^{1/3} 
\end{equation}
for all $z, z' \in \HbdS$. In particular, $\mf$ can be uniquely extended to a uniformly $1/3$-Hölder continuous function $\HbdSclosed \to \BF$,
which we also denote by $\mf$. 
\item 
The measure $\brho$ from \eqref{eq:mf_Stieltjes_representation} is absolutely continuous, i.e., there is a function 
$\brho^d\colon \Xf \times \R\setminus (-2\deltamf, 2\deltamf) \to [0,\infty), \; (x,\tau) \mapsto \brho^d_x(\tau)$ such that 
\begin{equation} \label{eq:decomposition_brho}
 \left(\brho_x|_{\R\setminus (-2\deltamf, 2\deltamf)}\right) (\di \tau) = \brho_x^d (\tau) \di \tau,\qquad\qquad \text{for all }x \in\Xf.
\end{equation}
The components $\brho_x^d$ are comparable with each other, i.e., $\brho_x^d(\tau) \sim \brho_y^d(\tau)$ for all $x, y \in \Xf$ and $\tau \in \R\setminus [-2\deltamf,2\deltamf]$. 
Moreover, the function $\brho^d\colon \R \setminus [-2\deltamf,2\deltamf] \to \BF$ is uniformly $1/3$-Hölder continuous, symmetric in $\tau$, $\brho^d(\tau) = \brho^d(-\tau)$,
and real-analytic around any $\tau \in \R\setminus [-2\deltamf, 2\deltamf]$ apart from points $\tau \in \supp \avg{\brho^d}$, where $\brho^d(\tau) =0$.
\end{enumerate}
\end{pro}

A similar result has been obtained in Theorem 2.4 in \cite{AjankiQVE} essentially relying on the uniform primitivity assumption \textbf{A3} in \cite{AjankiQVE}. 
For discrete $\Xfone$ and $\Xftwo$ without assuming~\mfbounded, Lemma 3.8 in \cite{AltGram} shows Hölder continuity of $\avg{\mf}$ instead of $\mf$ with a smaller exponent than $1/3$. 
Both conditions, \textbf{A3} in \cite{AjankiQVE} and the discreteness of $\Xfone$ and $\Xftwo$, are violated in our setup.
However, based on the proof of Theorem 2.4 in \cite{AjankiQVE}, we now explain how to extend the arguments of \cite{AjankiQVE} and \cite{AltGram} to show Proposition~\ref{pro:properties_of_mf}. 

\begin{lem} \label{lem:prop_mf_condition_mfbounded}
Uniformly for all $z \in \HbdS$, we have
\begin{align} 
\abs{\mf(z)} & \sim 1,\label{eq:mf_order_one} \\
 \Im \mf(z) & \sim \avg{\Im \mf(z)}. \label{eq:Im_mf_sim_avg_Im_mf}
\end{align}
\end{lem}

Using the arguments in the proof of Lemma 5.4 in \cite{AjankiQVE}, Lemma~\ref{lem:prop_mf_condition_mfbounded} follows immediately from 
\lowerbound, \mfbounded and \eqref{eq:QVE}. 
Here, as in the proof of Lemma~3.1 in~\cite{AltGram}, 
the uniform primitivity assumption \textbf{A3} of \cite{AjankiQVE} has to be replaced by (B') in \cite{AltGram}, 
which is a direct consequence of \lowerbound.

The Hölder continuity and the analyticity of $\mf$ and hence $\brho^d$ will be consequences of analyzing the perturbed~QVE 
\begin{equation} \label{eq:perturbed_QVE}
- \frac{1}{\gf} = z + \Sf\gf + \df 
\end{equation}
for $z \in \Hb$ 
and $\df = z - z'$ as well as the stability operator $\Bf$ defined through
\begin{equation} \label{eq:def_Bf}
 \Bf(z)\uf = \frac{\abs{\mf(z)}^2}{\mf(z)^2}\uf - \Ff(z)\uf, 
\end{equation}
where $\Ff(z) \colon \BF \to \BF$ is defined through $\Ff(z) \uf = \abs{\mf(z)} \Sf \left( \abs{\mf(z)} \uf\right)$ 
for any $\uf \in \BF$ (cf. \cite{AjankiQVE,AltGram}). 
Correspondingly, we introduce $F(z) \colon \BFtwo \to \BFone$ via $F(z) w = \abs{m_1(z)} S (\abs{m_2(z)} w )$ 
for $w \in \BFtwo$ and $F^t(z) \colon \BFone \to \BFtwo$ via $F^t(z) u = \abs{m_2(z)} S^t (\abs{m_1(z)} u )$ 
for $u \in \BFone$.

To formulate the key properties of $\Ff$ and $\Bf$, we now introduce some notation. 
The operator norms for operators on $\BF$ and $L^2(\meas)$ are denoted by $\norminf{\genarg}$ and $\normtwo{\genarg}$, respectively. 
If $T \colon L^2 \to L^2$ is a compact self-adjoint operator then the \emph{spectral gap} $\Gap(T)$ is the difference between the 
two largest eigenvalues of $\abs{T}$. 
We remark that $S$ and hence $FF^t$ are compact operators due to \Ltwoinf. 

\begin{lem}[Properties of $\Ff$] \label{lem:prop_Ff}
The eigenspace of $\Ff$ associated to $\normtwo{\Ff}$ is one-dimensional and spanned by a unique $L^2(\meas)$-normalized positive
$\ffp \in \BF$. The eigenspace associated to $-\normtwo{\Ff}$ is one-dimensional and spanned by $\ffm \defeq \ffp \emin\in\BF$. 
We have 
\begin{equation}
 \ffp \sim 1 \label{eq:ffp_sim_1}
\end{equation}
uniformly for $z \in \HbdS$. 
There is $\eps \sim 1$ such that 
\begin{equation} \label{eq:Gap_F} 
\normtwo{\Ff \uf} \leq (\normtwo{\Ff}-\eps) \normtwo{\uf}
\end{equation}
uniformly for $z \in\HbdS$ and for all $\uf \in \BF$ satisfying $\scalar{\ffp}{\uf} = 0$ and $\scalar{\ffm}{\uf} = 0$.
Furthermore, we have $\normtwo{\Ff}\leq 1$, $\Gap(F(z)F^t(z)) \sim 1$ uniformly for $z \in \HbdS$.
\end{lem}

Lemma~\ref{lem:prop_Ff} is a consequence of the proof of Lemma 3.3 in \cite{AltGram} with $r = \abs{\mf}$ and \eqref{eq:mf_order_one}.

\begin{lem} \label{lem:bound_Bf}
Uniformly for $z \in \HbdS$, we have
\begin{equation} \label{eq:correct_estimate_norm_Bf_inverse}
 \norminf{\Bf^{-1}(z)} \lesssim \frac{1}{\avg{\Im \mf(z)}^{2}}. 
\end{equation}
\end{lem}

\begin{proof} 
We describe the modifications in the proof of Lemma 3.5 in \cite{AltGram} necessary to obtain~\eqref{eq:correct_estimate_norm_Bf_inverse}. 
We remark that (3.10) in \cite{AltGram} holds true due to \Ltwoinf.

Let $z \in \HbdS$.
Taking the real part in \eqref{eq:QVE}, using \eqref{eq:mf_order_one} and Lemma \ref{lem:prop_Ff}, we obtain the bound $\normtwoa{\Re \mf\abs{\mf}^{-1}} \geq \abs{\Re z} \normtwo{\mf}/2 \gtrsim \abs{\Re z}$.
Therefore, using $\avg{(\Im \mf)^2} \geq \avg{\Im \mf}^2$ by Jensen's inequality, we obtain (3.27) in \cite{AltGram} with $\kappa = 2$. Employing $\Gap\left(F(z)F^t(z)\right) \sim 1$, 
we get  $\norminf{\Bf^{-1}(z)} \lesssim (\Re z)^{-2} \avg{\Im \mf(z)}^{-2}$. 
As $\normtwo{\Bf^{-1} (z) } \leq (1- \normtwo{\Ff(z)})^{-1} \lesssim (\Im z)^{-1}$ by (3.21) in \cite{AltGram} we conclude from $\Im \mf \lesssim \min\{1, (\Im z)^{-1}\}$ that 
$\norminf{\Bf^{-1}(z)} \lesssim \abs{z}^{-2} \avg{\Im \mf(z)}^{-2}$. 
This concludes the proof of \eqref{eq:correct_estimate_norm_Bf_inverse} since $\abs{z} \geq 2\deltamf$. 
\end{proof}

Note that if $\brho$ has a density $\brho^d$ around a point $\tau_0$ then,
uniformly for $\tau$ in a neighbourhood of $\tau_0$, we have 
\begin{equation} \label{eq:relation_Im_mf_brho}
 \brho^d(\tau) = \pi^{-1} \lim_{\eta \downarrow 0} \Im \mf(\tau + \ii \eta). 
\end{equation}

\begin{proof}[Proof of Proposition \ref{pro:properties_of_mf}]
Following the proof of Proposition 7.1 in \cite{AjankiQVE} yields the uniform $1/3$-Hölder continuity of $\mf$ and $\brho^d$. 
In this proof, the estimate (5.40b) has to be replaced by \eqref{eq:correct_estimate_norm_Bf_inverse}. Furthermore, 
\eqref{eq:Im_mf_sim_avg_Im_mf} substitutes Proposition 5.3 (ii) in \cite{AjankiQVE}, in particular, $\brho^d_x(\tau) \sim \brho^d_y(\tau)$. 
We remark that now the same proofs extend Lemma \ref{lem:prop_mf_condition_mfbounded}, Lemma \ref{lem:prop_Ff} and Lemma \ref{lem:bound_Bf} to all $z \in \HbdSclosed$. 
Hence, the proof of Corollary 7.6 in \cite{AjankiQVE} yields the analyticity using \eqref{eq:relation_Im_mf_brho} for $\tau \in \R \cap \HbdSclosed$.
\end{proof}

\subsection{Singularities of $\brho^d$ and the cubic equation}

We now study the behaviour of $\brho^d$ near points $\tau \in\R$, where $\brho^d$ is not analytic. 
Theorem 2.6 in \cite{AjankiQVE} describes the density near the edges and the cusps as well as the transition between the bulk and the 
singularity regimes in a quantitative manner. 
The same results hold for $\brho^d$ as well: 

\begin{pro} \label{pro:singularities_mf}
We assume \measprop, \lowerbound, \Ltwoinf and \mfbounded.
Then all statements of Theorem 2.6 in \cite{AjankiQVE} hold true on $\R\setminus[-2\deltamf,2\deltamf]$. 
\end{pro}

For the proof of Proposition \ref{pro:singularities_mf} we follow Chapter 8 and 9 in \cite{AjankiQVE} which contain the proof of the analogue of 
Proposition \ref{pro:singularities_mf}, Theorem 2.6 in \cite{AjankiQVE}, and describe the necessary changes as well as the main philosophy. 

The shape of the singularities of $\mf$ as well as the stability of the QVE (cf. Chapter 10 in \cite{AjankiQVE}) will be a consequence of the stability of a cubic equation.
We note that similar as in Lemma 8.1 of \cite{AjankiQVE}, the following properties of the stability operator $\Bf=\Bf(z)$ defined in \eqref{eq:def_Bf} can be proven. 
There is $\eps_*\sim 1$ such that for $z \in \HbdSclosed$ satisfying $\avg{\Im \mf(z)} \leq \eps_*$, $\Bf$ has a unique eigenvalue $\beta=\beta(z)$ of smallest modulus
and $\abs{\beta'}- \abs{\beta} \gtrsim 1$
for all $\beta' \in \spec(\Bf)\setminus \{ \beta \}$. The eigenspace associated to $\beta$ is one-dimensional and there is a unique vector $\bb=\bb(z) \in \BF$ 
in this eigenspace such that $\scalar{\bb(z)}{\ffp} = 1$.

Let $z \in \HbdSclosed$ such that $\avg{\Im \mf(z)} \leq \eps_*$ and $\gf\in \BF$ satisfy the perturbed QVE, \eqref{eq:perturbed_QVE}, at~$z$. 
We define 
\begin{equation} \label{eq:def_Theta_1}
\Theta(z) \defeq \scalara{\frac{\bar \bb(z)}{\avg{\bb(z)^2}}}{\frac{\gf -\mf(z)}{\abs{\mf(z)}}}. 
\end{equation}
By possibly shrinking $\eps_* \sim 1$, we obtain that if $\norminf{\gf-\mf(z)} \leq \eps_*$ then it can be shown as in Proposition 8.2 in \cite{AjankiQVE} that 
$\Theta$ satisfies 
\begin{equation} \label{eq:general_cubic}
 \mu_3 \Theta^3 + \mu_2 \Theta^2 + \mu_1 \Theta + \scalar{\abs{\mf}\bar \bb}{\df} = \kappa\left((\gf-\mf)/\abs{\mf}, \df\right), 
\end{equation}
where $\mu_1, \mu_2$ and $\mu_3$, which depend only on $S$ and~$z$, as well as $\kappa$ are given in \cite{AjankiQVE}. 

The main ingredient that needs to be changed in our setup is the estimate in (8.13) of \cite{AjankiQVE}. It gives a lower bound on  
the nonnegative quadratic form 
\begin{equation} \label{eq:def_mathcal_D} 
\Dquad(\wf) \defeq 
\scalara{\Qfp\wf}{\left(\normtwo{\Ff}+ \Ff\right)\left(1 - \Ff\right)^{-1} \Qfp\wf}
\end{equation}
for $\wf \in \BF$, where the projection $\Qfp$ is defined through $\Qfp \wf \defeq \wf - \scalar{\ffp}{\wf}\ffp$.
For some $c(z) >0$ and all $\wf \in \BF$, this lower bounds reads as follows
\begin{equation} \label{eq:QVE_paper_8.13}
\Dquad(\wf) \geq c(z) \normtwo{\Qfp\wf}^2.
\end{equation}
However, in our setup, owing to the second unstable direction $\ffm\perp\ffp$, $\Ff\ffm = -\normtwo{\Ff}\ffm$, we have $\Dquad(\ffm) = 0$ which contradicts \eqref{eq:QVE_paper_8.13}. 
In \cite{AjankiQVE}, the estimate \eqref{eq:QVE_paper_8.13} is only used to obtain 
\begin{equation}\label{eq:stability_cubic_general}
 \abs{\mu_3(z)} + \abs{\mu_2(z)} \gtrsim 1 
\end{equation}
(cf.  (8.34) in \cite{AjankiQVE}) for all $z \in\HbdSclosed$ satisfying $\avg{\Im \mf(z)} \leq \eps_*$ and $\norminf{\gf - \mf(z)} \leq \eps_*$ for $\eps_* \sim 1$ small enough.
In fact, it is shown above (8.50) in \cite{AjankiQVE} that  
\begin{equation} \label{eq:comparison_mu_3_mu_2_and_sigma_psi}
\abs{\mu_3} \gtrsim \psi + \ord(\alpha)\qquad \abs{\mu_2} \gtrsim \abs{\sigma} + \ord(\alpha).
\end{equation}
Here, we introduced the notations $\psi \defeq \Dquad(\pf \ffp^2)$ with $\pf \defeq \sign{\Re \mf}$ as well as $\alpha \defeq \avg{\ffp \Im \mf/\abs{\mf}}$ and $\sigma \defeq \scalar{\ffp}{\pf\ffp^2}$.
The proof used in \cite{AjankiQVE} to show \eqref{eq:comparison_mu_3_mu_2_and_sigma_psi} works in our setup as well.
Since $\alpha=\avg{\ffp\Im \mf/\abs{\mf}} \sim \avg{\Im \mf} \leq \eps_*$ by \eqref{eq:mf_order_one} 
and \eqref{eq:ffp_sim_1}, 
we conclude that $\abs{\mu_3} + \abs{\mu_2} \gtrsim \psi + \abs{\sigma}$ for $\eps_* \sim 1$ small enough. Hence, \eqref{eq:stability_cubic_general} is a consequence of 

\begin{lem}[Stability of the cubic equation]  \label{lem:stab_cubic}
There exists $\eps_* \sim 1$ such that 
\begin{equation}
 \psi(z) + \sigma^2(z) \sim 1 \label{eq:stability_cubic} 
\end{equation}
uniformly for all $z \in \HbdSclosed$ satisfying $\avg{\Im\mf(z) } \leq \eps_*$. 
\end{lem}

\begin{proof}
We first remark that due to \eqref{eq:mf_order_one}, \eqref{eq:Im_mf_sim_avg_Im_mf} and possibly shrinking $\eps_*\sim 1$ we can assume 
\begin{equation} \label{eq:real_part_sim_1}
 \abs{\Re \mf(z)} \sim 1
\end{equation}
for $z\in\HbdSclosed$ satisfying $\avg{\Im \mf(z)} \leq \eps_*$.
Second, owing to \eqref{eq:Gap_F}, for all $\wf \in\BF$, we have the following analogue of \eqref{eq:QVE_paper_8.13} 
\begin{equation} \label{eq:lower_bound_quad_form}
\Dquad(\wf) \gtrsim \normtwo{\Qfpm\wf}^2,
\end{equation}
where $\Qfpm$ is the projection onto the orthogonal complement of $\ffp$ and $\ffm$, i.e, $\Qfpm\wf = \wf - \scalar{\ffp}{\wf}\ffp - \scalar{\ffm}{\wf} \ffm$. 
Note that \eqref{eq:Gap_F} also yields the upper bound $\Dquad(\wf)\lesssim \normtwo{\Qfp\wf}^2$ and hence the upper bound in \eqref{eq:stability_cubic} by \eqref{eq:ffp_sim_1}.
Therefore, it suffices to prove the lower bound in \eqref{eq:stability_cubic}.
A straightforward computation starting from \eqref{eq:lower_bound_quad_form} and using $\ffm= \emin\ffp$ yields
\begin{equation} \label{eq:lower_bound_cubic_aux2}
 \psi +\sigma^2 = \mathcal D(\pf\ffp^2) + \avg{\pf\ffp^3}^2 \gtrsim \normtwo{\pf\ffp^2 - \scalar{\ffm}{\pf\ffp^2}\ffm}^2 =  \avga{\ffp^2\left(\pf\ffp - \avg{\pf\emin\ffp^3} \emin\right)^2}.
\end{equation}
Using \eqref{eq:ffp_sim_1}, \eqref{eq:real_part_sim_1} and $\abs{\Re \mf} = \pf \Re \mf$, we conclude
\begin{align}
\psi + \sigma^2 & \gtrsim \avga{ \left(\Re \mf\right)^2 \left( \pf\ffp  - \avg{\pf\emin\ffp^3} \emin\right)^2} \nonumber \\ 
 & \geq \avg{\ffp\abs{\Re \mf}}\left(\avg{\ffp\abs{\Re \mf}} + 2 \avg{\pf\emin\ffp^3} \avg{\emin}\Re \frac{1}{z} \right) \label{eq:lower_bound_cubic_aux6}
\end{align}
Here, we employed Jensen's inequality and \eqref{eq:scalar_emin_m} in the second step. 
Since $z \in \HbdSclosed$ and $\avg{\emin} = 0 $ for $\measone(\Xfone)=\meastwo(\Xftwo)$, there exists $\iota_* \sim 1$ such that 
the last factor on the right-hand side of \eqref{eq:lower_bound_cubic_aux6} is bounded from below by $\avg{\ffp\abs{\Re\mf}}/2$
for all $z \in \HbdSclosed$ and $\abs{\measone(\Xfone)-\meastwo(\Xftwo)} \leq \iota_* (\measone(\Xfone) + \meastwo(\Xftwo))$. 
Since $\avg{\ffp\abs{\Re \mf}}^2 \gtrsim 1$ by \eqref{eq:ffp_sim_1} and \eqref{eq:real_part_sim_1},
this finishes the proof of \eqref{eq:stability_cubic} for  $\abs{\measone(\Xfone)-\meastwo(\Xftwo)} \leq \iota_* (\measone(\Xfone) + \meastwo(\Xftwo))$.
For the proof of \eqref{eq:stability_cubic} in the remaining regime, $\abs{\measone(\Xfone)-\meastwo(\Xftwo)} > \iota_* (\measone(\Xfone) + \meastwo(\Xftwo))$, we introduce
 $\yf\defeq \emin \pf \ffp$ and use $\yf^2 = \ffp^2 \sim 1$ and $\left(\yf+ \avg{\yf^3}\right)^2 \lesssim 1$ by \eqref{eq:ffp_sim_1} to obtain 
 from \eqref{eq:lower_bound_cubic_aux2} the bound
\begin{equation} \label{eq:lower_bound_cubic_aux4}
 \psi +\sigma^2 \gtrsim  \avga{\left(\yf-\avg{\yf^3}\right)^2\left(\yf+\avg{\yf^3}\right)^2}
=\avga{\left((\yf^2-1)+(1-\avg{\yf^3}^2)\right)^2}\geq \avga{\left(\yf^2-1\right)^2}.
\end{equation}
Here, we used $\avg{\yf^2} = \avg{\ffp^2} = 1$ and $\left(1- \avg{\yf^3}^2\right)^2 \geq 0$. 
Since $0 = \scalar{\ffm}{\ffp} = \avg{\emin\yf^2}$, using \eqref{eq:lower_bound_cubic_aux4}, we conclude
\begin{equation} \label{eq:lower_bound_cubic_aux1}
 \avg{\emin}^2 = \avg{\emin(1-\yf^2)}^2 \leq \avg{(1-\yf^2)^2} \lesssim \psi + \sigma^2. 
\end{equation}
This implies \eqref{eq:stability_cubic} for $\abs{\measone(\Xfone)-\meastwo(\Xftwo)} > \iota_* (\measone(\Xfone) + \meastwo(\Xftwo))$ as $\avg{\emin}^2 \geq \iota_*^2 \sim 1$. 
This completes the proof of Lemma~\ref{lem:stab_cubic}.
\end{proof}

Following the remaining arguments of chapter 8 and 9 in \cite{AjankiQVE} yields Proposition \ref{pro:singularities_mf}.

\section{Proofs of Theorem~\ref{thm:prop_selfcon_density_states} and Theorem~\ref{thm:local_law}}

\begin{proof}[Proof of Theorem \ref{thm:prop_selfcon_density_states}]
By Remark~\ref{rem:Gram_assumptions_to_QVE}, we can apply Proposition \ref{pro:properties_of_mf} for each $\deltamf>0$.
Hence, there are $\brho^0 \in \BF$ and $\brho^d \colon \Xf \times \R\setminus \{0\} \to [0,\infty)$ such that 
\[ \brho_x(\di \tau ) = \brho^0_x \delta_0(\di\tau) + \brho^d_x(\tau) \di \tau \]
for all $x \in\Xf$.  For $k\in \Xfone$, we set $\dens^0_k \defeq \brho^0_k$ and
\begin{equation} \label{eq:dens_x_d}
 \dens_k^d(E) \defeq E^{-1/2} \brho_k^d(E^{1/2}) \char(E >0)
\end{equation}
with $E \in \R$.
Therefore, using \eqref{eq:transform_QVE_to_Gram}, we obtain~\eqref{eq:def_dens_0_dens_d} (cf. the proof of Theorem 2.1 in~\cite{AltGram}).
The $1/3$-Hölder continuity of $\brho^d$ implies the $1/3$-Hölder continuity of 
$\dens^d$. Similarly, the analyticity of $\dens^d$ is obtained from the analyticity of $\brho^d$. 
From Proposition~\ref{pro:singularities_mf} with $\deltamf= \sqrt{\delta}/2$, we conclude that $\posset\cap (\delta, \infty)$ is a finite union of open intervals and 
its connected components have a Lebesgue measure of at least $2\rho_*$ for some $\rho_*$ depending only on the model parameters and $\delta$. 
This completes the proof (i). 

For the proof of (ii), we follow the proof of Theorem 2.6 in \cite{AjankiCPAM}. We replace the estimates (4.1), (4.2), (5.3) and (6.7) as well as their proofs in \cite{AjankiCPAM} by 
 \eqref{eq:mf_order_one}, \eqref{eq:Im_mf_sim_avg_Im_mf}, \eqref{eq:correct_estimate_norm_Bf_inverse} and \eqref{eq:stability_cubic} as well as their proofs in this note, respectively. 
This proves a result corresponding to Theorem 2.6 in \cite{AjankiCPAM} for $\brho^d$ and $\tau_0 \in (\pt\posset) \cap (0,\infty)$ in our setup. 
Using the transform \eqref{eq:dens_x_d} completes the proof of Theorem~\ref{thm:prop_selfcon_density_states}.
\end{proof}

\begin{proof}[Proof of Theorem~\ref{thm:local_law}]
Note that \ranupbound implies \Ltwoinf.
By Remark \ref{rem:Gram_assumptions_to_QVE}, \mbounded implies \mfbounded.
Using \eqref{eq:stability_cubic_general} to replace (8.34) in \cite{AjankiQVE}, we obtain an analogue of Proposition 10.1 in~\cite{AjankiQVE} in our setup on $\HbdSclosed$. 
Therefore, we have proven in our setup analogues of all the ingredients provided in \cite{AjankiQVE} and used in \cite{Ajankirandommatrix} to prove a local law for Wigner-type random matrices with a 
uniform primitive variance matrix.
Thus, following the arguments in \cite{Ajankirandommatrix}, we obtain a local law for the resolvent of $\Hf$ defined in \eqref{eq:def_Hf} and spectral parameters 
$z \in \HbdS\cap \{ w \in \Hb \colon \Im w \geq (p+n)^{-1+\gamma}\}$, where $\deltamf = \sqrt{\delta}/2$ and $\gamma\in(0,1)$. 
Proceeding as in the proof of Theorem 2.2 in \cite{AltGram} yields Theorem~\ref{thm:local_law}.
\end{proof}




\providecommand{\bysame}{\leavevmode\hbox to3em{\hrulefill}\thinspace}


\end{document}